\pdfoutput=1
\documentclass[a4paper,12pt,oneside]{amsart} 

\usepackage{typearea} 
\typearea{13} 

\usepackage[UKenglish]{babel}
\usepackage[T1]{fontenc}

\usepackage{amsmath}
\usepackage{amsthm}
\usepackage{amsfonts}
\usepackage{amssymb}
\usepackage[neverdecrease]{paralist}
\usepackage{enumitem}
\usepackage{dsfont}
\usepackage{mathtools}
\usepackage{etoolbox}

\usepackage[hyphens,spaces]{url}

\usepackage{microtype}
\usepackage{hyperref}

\newcommand{\RR}{\mathbb{R}}
\newcommand{\ZZ}{\mathbb{Z}}
\newcommand{\NN}{\mathbb{N}}
\newcommand{\CC}{\mathbb{C}}

\newcommand{\cR}{\mathcal{R}}
\newcommand{\cS}{\mathcal{S}}
\newcommand{\cK}{\mathcal{K}}
\newcommand{\cH}{\mathcal{H}}
\newcommand{\cG}{\mathcal{G}}
\newcommand{\cA}{\mathcal{A}}
\newcommand{\cU}{\mathcal{U}}
\newcommand{\cF}{\mathcal{F}}
\newcommand{\cC}{\mathcal{C}}

\newcommand{\Linop}{\mathcal{L}}
\newcommand{\clos}[1]{\overline{#1}}
\newcommand{\conj}[1]{\overline{#1}}
\newcommand{\restrict}[2]{\ensuremath{#1|_{#2}}}
\newcommand{\emphdef}[1]{\textbf{\boldmath #1\unboldmath}}
\newcommand{\m}{\ensuremath{\text{m}}}
\newcommand{\wto}{\rightharpoonup}
\newcommand{\eps}{\varepsilon}

\DeclareMathOperator{\card}{card}
\DeclareMathOperator{\linspan}{span}
\DeclareMathOperator{\rg}{rg}
\DeclarePairedDelimiter\norm{\lVert}{\rVert}
\DeclarePairedDelimiter\abs{\lvert}{\rvert}

\newcommand{\scalar}[2]{\left(#1\mkern3mu{\mid}\mkern3mu #2\right)}

\newlist{alenum}{enumerate}{1}
\setlist[alenum]{label=\textup{(\alph*)},ref=\textup{(\alph*)},itemsep=0em,topsep=1ex}

\newlist{romanenum}{enumerate}{1}
\setlist[romanenum]{label=\textup{(\roman*)},ref=\textup{(\roman*)},itemsep=0em,topsep=1ex}

\theoremstyle{plain}
\newtheorem{theorem}{Theorem}[section]
\newtheorem{proposition}[theorem]{Proposition}
\newtheorem{corollary}[theorem]{Corollary}
\newtheorem{lemma}[theorem]{Lemma}
\theoremstyle{definition}
\newtheorem{example}[theorem]{Example}
\newtheorem{remark}[theorem]{Remark}
\newtheorem{definition}[theorem]{Definition}

\makeatletter 
\newcommand\forceparpenalty{\@beginparpenalty=10000}
\makeatother

\title[Nonseparability and von Neumann's theorem]{Nonseparability and von Neumann's theorem\\ for domains of unbounded operators}
\author[A.F.M. ter~Elst]{\sc A.F.M. ter~Elst}
\address{A.F.M. ter~Elst\\Department of Mathematics\\The University of Auckland\\Private bag 92019\\Auckland 1142\\New Zealand}
\email{terelst@math.auckland.ac.nz}
 
\author[M. Sauter]{\sc Manfred Sauter}
\address{Manfred Sauter\\Institute of Applied Analysis\\Ulm University\\89069 Ulm\\Germany}
\email{manfred.sauter@uni-ulm.de}

\date{April 2015}


\begin{document}

\begin{abstract}
A classical theorem of von Neumann asserts that every unbounded self-adjoint operator $A$ 
in a \emph{separable} Hilbert space $H$ is unitarily equivalent to an operator $B$ in $H$ such that $D(A)\cap D(B)=\{0\}$.
Equivalently this can be formulated as a property for nonclosed operator ranges.
We will show that von Neumann's theorem does not directly extend to the nonseparable case. 

In this paper we prove a characterisation of the property that an operator range $\cR$ in a general Hilbert space $H$ 
admits a unitary operator $U$ such that $U\cR\cap\cR=\{0\}$. This 
allows us to study stability properties of operator ranges with the
aforementioned property.
\end{abstract}

\begingroup
\makeatletter
\patchcmd{\@settitle}{\uppercasenonmath\@title}{\large}{}{}
\patchcmd{\@setauthors}{\MakeUppercase}{}{}{}
\makeatother
\maketitle
\endgroup

\section{Introduction}

In classical works like~\cite{vNeu30:herm-op} or~\cite{Nagy42}, Hilbert spaces
were introduced as \emph{separable} complete inner product spaces. 
The notion of `separability' is due to Fr\'echet~\cite[p.\,23]{Fre06}, 
likely originating in the property that the rationals `separate' the reals.
Early works that generalise the Hilbert space theory to the nonseparable case are~\cite{Loew34} and~\cite{Rel34}.

While separability frequently allows for simplified proofs and an effective approximation using a specific countable set of
elements,
for Hilbert spaces the assumption of separability is often only made for convenience,
and the results frequently hold -- with the appropriate changes -- also in the nonseparable setting.
In the present paper we investigate von Neumann's theorem about the domains of unbounded self-adjoint operators:
\begin{theorem}[{\cite[Satz 18]{vNeu29}}]\label{thm:vNeu-dom}
Let $H$ be a \emph{separable} Hilbert space and $A$ an unbounded self-adjoint operator in $H$.
Then there exists a unitary operator $U$ such that $D(U^*AU)\cap D(A) = \{0\}$.
\end{theorem}

We show in Example~\ref{ex:ce-vNeu-dom} that a naive reformulation for the nonseparable case is false, and provide
an appropriate generalisation that works for general Hilbert spaces in Theorem~\ref{thm:vNeu-nonsep}.
For our arguments we employ Dixmier's approach on von Neumann's theorem as presented in~\cite[Section~3]{FW71:op-rg}.

Related questions in the much more diverse Banach space setting, still with the separability assumption,
however, have been treated in~\cite{CS98:disjoint}.
Very recently, sharper versions of von Neumann's theorem related to
Schm\"udgen's theorem~\cite[Theorem~5.1]{Schm83} and involving the domains of fractional powers were
presented in~\cite{AZ2015}, again in the separable case.
In~\cite{Kos2006} concrete examples of operators were given with certain
intersection properties of the fractional domains.

Since separable Hilbert spaces are frequently the most important case for applications, often only the separable case is
considered, which sometimes helps to simplify the exposition.
Of course there are also problems that are much easier in the nonseparable case, an example being the famous invariant subspace problem.

There are numerous instances, where some result was only substantially later extended to the nonseparable case.
In these cases usually a suitable reformulation of the problem was required.
Examples include the extension of the spectral theorem~\cite{Rel34,Loew34},
the characterisation of closed two-sided ideals~\cite{Luft68},
the description of the distance of an operator to the set of unitary operators~\cite{tElst90:unit},
Gleason's theorem~\cite{EH75,Sol2009:preprint} and~\cite[Section~3.5, for example]{Dvu93},
the block diagonalisation of general operators~\cite{Mik2009}
or a generalisation of semi-Fredholm operators~\cite{Bou95}.

Naturally there are plenty of cases, where results have been only established in the separable case,
for example the unitarily invariant classification of operator ranges in~\cite{LT76} and~\cite{Dix49:var-J}. 
In~\cite{Dix49:var-J} the author writes about the general nonseparable case:
\begin{quote}
Les cas g\'en\'eral peut aussi se traiter, mais conduit \`a des classifications plus 
p\'enibles, les questions de dimension jouant souvent un r\^ole essentiel.
\end{quote}
But there are also open problems specifically for the nonseparable case, see for example~\cite{FMcKS2013}.

A short outline of this paper is as follows. In Section~\ref{sec:ce} we present a counterexample showing that
von Neumann's theorem does not directly extend to nonseparable Hilbert spaces.
Then we gather required prerequisites about operator ranges in
Section~\ref{sec:or}.
Our reformulation of von Neumann's theorem for general Hilbert spaces, which is the main result of this paper, is proved in Section~\ref{sec:vNeu-gen}.
In the final Section~\ref{sec:stab} we apply our reformulation to obtain stability and density properties, closing with a curious
counterexample in Example~\ref{ex:ce-dens}.

We assume that the reader is familiar with several basic facts from set theory and the arithmetic of cardinal numbers.
For the required background we refer to~\cite[Section I.3, in particular~(3.14)]{Jech2003}.
In this paper $\dim H$ refers to the Hilbert space dimension of a Hilbert space $H$,
i.e.\ to the cardinality of one/every orthonormal basis of $H$.
Given a Hilbert space $H$, the set of unitary operators on $H$ will be denoted by $\cU$.
Throughout we will work in the theory ZFC.
Moreover, the Hilbert spaces considered here are assumed to be complex.

\section{A counterexample}\label{sec:ce}

In this section we provide a counterexample to von Neumann's theorem in a nonseparable Hilbert space.
We need the following lemma that allows to compare the dimension of two Hilbert spaces.
\begin{lemma}\label{lem:cmp-dim}
Let $H$ and $K$ be Hilbert spaces and $j\in\Linop(H,K)$.
Then $\dim\clos{\rg j}\le\dim H$.
If $j$ is in addition injective, then $\dim H=\dim\clos{\rg j}$.
\end{lemma}
\begin{proof}
The first part is shown in~\cite[Lemma~2.4]{Luft68}.

Now suppose that $j$ is injective. Consider the polar decomposition of $j$. So $j=UP$,
where $P$ is a positive semi-definite operator on $H$ and $U$ is a partial isometry with initial space $H$ and
final space $\clos{\rg j}$. The latter implies that $\dim H=\dim\clos{\rg j}$.
\end{proof}
Note that in Lemma~\ref{lem:cmp-dim} the continuity of $j$ is essential, of course, as $\ell^2(\NN)$ and $\ell^2(\RR)$ have the same vector space dimension~\cite{Lacey73}.

The following example shows that Theorem~\ref{thm:vNeu-dom} does not directly extend to
the nonseparable case.
\begin{example}\label{ex:ce-vNeu-dom}
Let $H\coloneqq H_1\oplus\ell_2(\NN)$, where $H_1$ is a nonseparable Hilbert space. Then
\begin{equation}\label{eq:ex-dim-ineq}
    \dim H_1>\dim\ell_2(\NN)=\aleph_0.
\end{equation}
Let $T$ be the unbounded multiplication operator in $\ell_2(\NN)$ given by 
\[
    D(T)=\Bigl\{a\in\ell_2(\NN) : \sum_{n=1}^\infty 4^n\abs{a_n}^2 < \infty\Bigr\}
\]
and $Te_n = 2^ne_n$ for all $n\in\NN$.
Define the operator $A$ in $H$ by $A=I\oplus T$, where $I$ is the identity operator on $H_1$.
Then $A$ is an unbounded self-adjoint operator.

Let $U$ be a unitary operator on $H$. By~\eqref{eq:ex-dim-ineq} and Lemma~\ref{lem:cmp-dim}
the operator $P_2 \restrict{U}{H_1}\colon H_1\to\ell_2(\NN)$ is not injective,
where $P_2$ is the projection onto the second component in $H=H_1\oplus\ell_2(\NN)$.
So there exist $x\in H_1\setminus\{0\}$ and $y\in H_1$ such that $U(x,0) =(y,0)$.
Hence $D(A)\cap D(U^*AU)\ne \{0\}$.
In particular, there does not exist an operator $B$ that is unitarily equivalent to $A$ and satisfies $D(A)\cap D(B)=\{0\}$.
\end{example}

\section{Operator ranges}\label{sec:or}

In this section we make use of Dixmier's technique~\cite{Dix49:etude-julia} as presented in \cite{FW71:op-rg}.
We recall basic properties of \emph{operator ranges}
and consider an equivalent reformulation of von Neumann's theorem in terms of operator ranges.
Moreover, we compare the operator range of a bounded operator with that of its adjoint.
\begin{definition}
Let $H$ be a Hilbert space. The vector subspaces that are the range of a bounded
operator on $H$ are called \emphdef{operator ranges} in $H$.
\end{definition}

As the following example shows, an operator range obviously does not need to be closed.
We shall see that the operator range in the example is in some sense canonical. 
\begin{example}\label{ex:T-nc-range}
Let $H=\ell_2(\NN)$. Define $T\in\Linop(H)$ by $Te_n=2^{-n} e_n$ for all $n\in\NN$, where $(e_n)_{n\in\NN}$ is the usual orthonormal basis of $\ell^2(\NN)$.
Then $\rg T$ is dense in $H$, but not closed as  $(1,\tfrac{1}{2},\tfrac{1}{4},\tfrac{1}{8},\ldots)$ is not contained in $\rg T$.
\end{example}

A straightforward reformulation of von Neumann's theorem in terms of operator ranges is as follows.
We point out that the proof for this reformulation in~\cite[Theorem~3.6]{FW71:op-rg} based on Dixmier's technique is
completely different and considerably less involved than von~Neumann's original proof~\cite[Satz~18]{vNeu29}.
\begin{theorem}\label{thm:vNeu-op-ran}
Let $H$ be a separable Hilbert space. If $\cR$ is a nonclosed operator range in $H$, then there exists
a unitary operator $U$ such that $U\cR\cap \cR=\{0\}$.
\end{theorem}

While the naive extension of Theorem~\ref{thm:vNeu-op-ran} to the nonseparable case is false, 
one can actually give useful characterisations of the operator ranges $\cR$ for which such a unitary operator $U$ exists.
To this end we need a better understanding of operator ranges.

\begin{proposition}\label{prop:opran}
Let $H$ be a Hilbert space.
\begin{alenum}
\item\label{en:or-cl} 
Every closed subspace of $H$ is an operator range in $H$.
\item\label{en:or-psa} 
Every operator range in $H$ is the range of a positive operator in $\Linop(H)$.
\item\label{en:or-hi} 
A vector subspace $\cR$ of $H$ is an operator range in $H$ if and only if $\cR$ can be 
equipped with a complete inner product such that it is continuously embedded into $H$.
\item\label{en:or-sum} 
The vector sum of two operator ranges is an operator range.
In fact, if $T,S\in\Linop(H)$,
then $\rg T+\rg S = \rg(TT^* + SS^*)^{1/2}$.
\item\label{en:or-is}
The intersection of two operator ranges is an operator range. 
\item\label{en:or-sum-all}
If $\cR$ and $\cS$ are operator ranges in $H$ such that $\cR+\cS=H$, then
there exist closed subspaces $M_1,M_2$ of $H$ with $M_1\subset\cR$, $M_2\subset\cS$, $M_1\cap M_2=\{0\}$ and $M_1+M_2=H$.
\item\label{en:or-sum-clos}
If $\cR$ and $\cS$ are operator ranges in $H$ such that $\cR\cap\cS=\{0\}$ and $\cR+\cS$ is closed,
then both $\cR$ and $\cS$ are closed.
\end{alenum}
\end{proposition}
\begin{proof}
Statement~\ref{en:or-cl} follows by using the corresponding orthogonal projection in $H$.
To prove \ref{en:or-psa}, it suffices to note that $\rg T=\rg (TT^*)^{1/2}$, which follows from Douglas'
lemma~\cite[Theorem~2.1]{FW71:op-rg}.
Statement~\ref{en:or-hi} is part of~\cite[Theorem~1.1]{FW71:op-rg}, \ref{en:or-sum} is given
in~\cite[Theorem~2.2]{FW71:op-rg} and~\ref{en:or-is} is a consequence of~\cite[Corollary~2 after
Theorem~2.2]{FW71:op-rg}.
Finally,~\ref{en:or-sum-all} can be found in~\cite[Theorem~2.4]{FW71:op-rg} 
and~\ref{en:or-sum-clos} follows from the proof of~\cite[Theorem~2.3]{FW71:op-rg}.
\end{proof}

Hence the operator ranges in $H$ are a lattice with respect to intersection and sum.
Note that the sum of two closed subspaces is not closed in general.
Moreover, not every vector subspace is an operator range.
\begin{example}
The following are examples of vector subspaces that are not operator ranges.
\forceparpenalty
\begin{enumerate}
\item The kernel of an unbounded linear functional $\varphi\colon H\to \CC$ is a dense, nonclosed vector subspace of $H$
with codimension $1$, but not an operator range in $H$. In fact, there exists an $x\in H\setminus\{0\}$ such that
$H=\ker\varphi+\linspan\{x\}$ and $\ker\varphi\cap\linspan\{x\}=\{0\}$. So $\ker\varphi$ cannot be an operator range in $H$ by Proposition~\ref{prop:opran}\,\ref{en:or-sum-clos}.
\item
The space $L^p(0,1)$ with $p>2$ is a subspace of $L^2(0,1)$, but not an operator range in $L^2(0,1)$; see~\cite[last paragraph on p.\,257]{FW71:op-rg}.
\item
If $A$ is a maximal accretive operator in $H$ that is not \m-accretive, then $\rg(I+A)$ is not an operator range in $H$,
see~\cite[Theorem~5.4 and Proposition~5.12]{tESV:arxiv}.
\end{enumerate}
\end{example}

The description of operator ranges in the next lemma will be essential for this paper.
\begin{lemma}[{\cite[Theorem~1.1\,(5)]{FW71:op-rg}}]\label{lem:opran-seq}
A vector subspace $\cR$ of $H$ is an operator range in $H$ if and only if
there exists a sequence of closed pairwise orthogonal subspaces $(\cH_n)$ such that
\[
    \cR = \Biggl\{\sum_{n=1}^\infty x_n : \text{$x_n\in\cH_n$ for all $n\in\NN$ and $\sum_{n=1}^\infty 4^n\norm{x_n}^2 <\infty$}\Biggr\}.
\]
\end{lemma}
\begin{proof}
This follows from Proposition~\ref{prop:opran}\,\ref{en:or-psa} and the spectral theorem.
\end{proof}

Adopt the notation of Lemma~\ref{lem:opran-seq}. We say that the sequence $(\cH_n)$ \emphdef{represents} the operator range $\cR$.
Note that $\cH_n=\{0\}$ for an $n\in\NN$ is allowed. Moreover, the sequence representing $\cR$ is not unique in general. 
For example, replacing $\cH_1$ and $\cH_2$ by $\{0\}$ and $\cH_1\oplus\cH_2$ (or vice versa) does not change the represented operator range $\cR$.

\begin{lemma}\label{lem:or}
Let $H$ be a Hilbert space. Let $\cR$ and $\cS$ be operator ranges in $H$.
Suppose that the sequence $(\cH_n)$ represents $\cR$ and that $(\cK_n)$ represents $\cS$.
\begin{alenum}
\item\label{en:lor-dense} 
One has $\clos{\cR}=\bigoplus_{k=1}^\infty\cH_k$ and $H=\cR^\perp \oplus\bigoplus_{k=1}^\infty\cH_k$. In particular, $\cR$ is dense in $H$ if and only if $H=\bigoplus_{k=1}^\infty\cH_k$.
Moreover, $\cR^\perp\oplus\bigoplus_{k=n+1}^\infty\cH_k=\bigl(\bigoplus_{k=1}^n\cH_k\bigr)^\perp$ for all $n\in\NN$.
\item\label{en:lor-unequi} 
If $\dim\cR^\perp = \dim\cS^\perp$ and $\dim\cH_n=\dim\cK_n$ for all $n\in\NN$, then there exists a unitary operator $U$ such that $U\cR=\cS$.
\item\label{en:lor-clos}
The operator range $\cR$ is closed if and only if there exists an $n_0\in\NN$ such that $\cH_n=\{0\}$ for all $n>n_0$.
\item\label{en:lor-contain} If $\cK_n\subset\bigoplus_{k=1}^n \cH_k$ for all $n\in\NN$, then $\cS\subset\cR$.
\end{alenum}
\end{lemma}
\begin{proof}
Statements~\ref{en:lor-dense}, \ref{en:lor-unequi} and \ref{en:lor-clos} are easy.

\ref{en:lor-contain}: 
Let $x\in\cS$.
By Lemma~\ref{lem:opran-seq} one can write $x=\sum_{n=1}^\infty x_n$,
where $x_n\in\cK_n$ for all $n\in\NN$ and $\sum_{n=1}^\infty 4^n\norm{x_n}^2<\infty$.
By the assumption it follows that for all $n\in\NN$ one can uniquely write
\[
    x_n = x^{(n)}_1+\dots+x^{(n)}_n
\]
with $x^{(n)}_k\in\cH_k$ for all $k\in\{1,\dots,n\}$.
Then
\begin{equation}\label{eq:def-yk}
\sum_{k=1}^\infty 4^k\sum_{n=k}^\infty \norm[\big]{x^{(n)}_k}^2 =
\sum_{n=1}^\infty\sum_{k=1}^n 4^k\norm[\big]{x^{(n)}_k}^2
	\le \sum_{n=1}^\infty 4^n \norm{x_n}^2 < \infty.
\end{equation}
For all $k\in\NN$ define $y_k := \sum_{n=k}^\infty x^{(n)}_k$, which converges due to~\eqref{eq:def-yk}. Moreover,
$y_k\in\cH_k$ for all $k\in\NN$ and
\[
	\sum_{k=1}^\infty y_k = \sum_{k=1}^\infty \sum_{n=k}^\infty x^{(n)}_k = \sum_{n=1}^\infty \sum_{k=1}^n x^{(n)}_k = x.
\]
Furthermore, by~\eqref{eq:def-yk} one has
\[
	\sum_{k=1}^\infty 4^k \norm{y_k}^2 
	\le \sum_{n=1}^\infty 4^n \norm{x_n}^2 < \infty.
\]
Therefore $x\in\cR$ by Lemma~\ref{lem:opran-seq}.
\end{proof}
\begin{remark}
Putting parts of $\cH_n$ into later spaces $\cH_k$ with $k\ge n$ potentially makes the represented operator range smaller.
Conversely, 
putting parts of $\cH_n$ into earlier spaces $\cH_k$ with $k\le n$ for infinitely many $n\in\NN$ potentially makes the represented operator range larger.
Both of these statements formally follow from Lemma~\ref{lem:or}\,\ref{en:lor-contain}.
\end{remark}

The next lemma compares the ranges of an operator and its adjoint. The following example highlights the main difficulty.
\begin{example}
Let $H=\ell^2(\NN)$ and let $(e_k)$ be the usual orthonormal basis. Let $A\in\Linop(H)$ be given by $Ae_k = e_{2k}$ for
all $k\in\NN$.
Note that $A$ is a partial isometry with initial space $H$ and final space $\clos{\linspan}\{e_{2k}:k\in\NN\}$.
It is obvious that there exists a unitary operator $U$ such that $U\rg A\cap\rg A=\{0\}$, but such an operator does not exist for $\rg A^*=H$.
\end{example}

\begin{lemma}\label{lem:opran-adj}
Let $A\in\Linop(H)$ and suppose $(\cH_n)$ represents $\rg A$.
Then there exists an orthogonal sequence $(\cK_n)$ representing $\rg A^*$ such that
$\dim\cK_n=\dim\cH_n$ for all $n\in\NN$.
Moreover, there exists a unitary operator $U$ such that $U\rg A=\rg A^*$ if and only if $\dim\ker A=\dim\ker A^*$.
\end{lemma}
\begin{proof}
Let $A=VP$ be the polar decomposition of $A$; i.e.\ $P$ is a positive operator and $V$ is a partial isometry
on $H$ with initial space $(\ker A)^\perp$ and final space $(\ker A^*)^\perp$.
By~\cite[Formulas after Theorem~7.2, p.\,138]{Schm2012} one has $\rg A^* = V^*\rg A$.
In particular, $\restrict{V^*}{(\ker A^*)^\perp}\colon (\ker A^*)^\perp \to (\ker A)^\perp$ is a unitary map that maps $\rg A$ onto $\rg A^*$.

For all $n\in\NN$ define $\cK_n := V^*\cH_n$. Since $\cH_n\subset\rg A\subset(\ker A^*)^\perp$, it follows that
$\restrict{V^*}{\cH_n}$ is an isometry and hence $\dim\cK_n=\dim\cH_n$ for all $n\in\NN$.
The sequence $(\cK_n)$ represents an operator range, which is the image under $V^*$ of the operator range represented by
$(\cH_n)$, i.e.\ $\rg A^*$.

Now we prove the second statement.
First suppose that $U\rg A = \rg A^*$ for a unitary operator $U$. Then
\[
    \ker A^* = (\rg A)^\perp = \bigl(\rg (U^*A^*)\bigr)^\perp = \ker(AU) = U^*\ker A.
\]
Hence $\dim\ker A^* = \dim\ker A$.

Conversely, suppose that $\dim\ker A=\dim\ker A^*$. Using this and the first statement of the lemma, it follows from Lemma~\ref{lem:or}\,\ref{en:lor-unequi} that there exists a unitary operator $U$ as claimed.
\end{proof}

The following is now a straightforward consequence.
\begin{proposition}
Let $H$ be a Hilbert space and $A$ a densely defined closed operator in $H$.
Suppose that $\rho(A)\ne\emptyset$.
Then there exists a unitary operator $U$ such that $D(U^*AU) \cap D(A)=\{0\}$
if and only if there exists a unitary operator $V$ such that $D(V^*A^*V)\cap D(A^*)=\{0\}$. 
\end{proposition}
\begin{proof}
Suppose that $\lambda\in\rho(A)$ and define $B := (\lambda I - A)^{-1}\in\Linop(H)$.
Then $\rg B = D(A)$, $B^*=(\conj{\lambda}I-A^*)^{-1}$ and $\rg B^* = D(A^*)$.
Moreover $\ker B = \ker B^* = \{0\}$. 
By Lemma~\ref{lem:opran-adj} there exists a unitary operator $W$ such that $W\rg B = \rg B^*$.
Now suppose that $U$ is a unitary operator such that $D(U^*AU)\cap D(A)=\{0\}$.
As $D(U^*AU)=U^*D(A)=U^*\rg B$, one has $U^*\rg B \cap \rg B = \{0\}$.
Then $V := WUW^*$ is a unitary operator and
\[
    V^*\rg B^*\cap \rg B^* = WU^*\rg B\cap W\rg B = W(U^*\rg B\cap\rg B)=\{0\}.
\]
As $V^*\rg B^* = V^*D(A^*) = D(V^*A^*V)$, the operator $V$ has the asserted property.
The converse direction follows from swapping the roles of $B$ and $B^*$.
\end{proof}

\section{Von Neumann's theorem for general Hilbert spaces}\label{sec:vNeu-gen}

To be self-contained later on, we start this section by presenting a proof along the lines of~\cite[Theorem and Corollary, p.\,520]{Isr2004} for the following 
variant of Theorem~\ref{thm:vNeu-op-ran}. Note that we do not assume that $H$ is separable, but as in~\cite{Isr2004} we make the strong
assumption that the operator range is the range of a compact operator.
\begin{proposition}\label{prop:israel}
Let $H$ be an infinite-dimensional Hilbert space and $T\in\Linop(H)$ be a compact operator.
Then the set 
\[
    \cG := \{U\in\cU : U\rg T\cap \rg T=\{0\}\}
\]
is a dense $G_\delta$ set in $\cU$ with respect to the uniform, strong and weak operator topology.
In particular, $\cG$ is not empty.
\end{proposition}
\begin{proof}
It suffices to prove that $\cG$ is a $G_\delta$ set with respect to the strong operator topology and
dense in $\cU$ with respect to the uniform operator topology.
It is easily seen that the weak operator topology on $\cU$ agrees with the strong operator topology, see also~\cite[Remark~4.10]{Tak2002}).

Define
$K_k := \clos{TB(0,k)}\setminus B(0,1/k)$ 
and $\cG_k := \{U\in\cU : UK_k\cap K_k=\emptyset\}$
for all $k\in\NN$. Then $K_k$ is a compact set for all $k\in\NN$ since $T$ is compact.
One has $\bigcup_{k\in\NN} K_k=\rg T\setminus\{0\}$.
In fact, the inclusion $\rg T\setminus\{0\}\subset\bigcup_{k\in\NN} K_k$ is obvious.
For the other direction, fix a $k\in\NN$ and suppose that $(y_n)$ is a sequence in $TB(0,k)\setminus B(0,1/k)$ such that $y_n\to y$ for a $y\in H$.
Then $y\ne 0$ and there exists a sequence $(x_n)$ in $B(0,k)$ such that $Tx_n=y_n$ for all $n\in\NN$. After passing to a subsequence, we may assume that
there exists an $x\in H$ such that $x_n\wto x$. As $T$ is compact, it follows that $y = \lim_{n\to\infty}
y_n=\lim_{n\to\infty} Tx_n = Tx$. So $y\in\rg T\setminus\{0\}$.
The identity $\bigcup_{k\in\NN} K_k=\rg T\setminus\{0\}$ and
the monotonicity $K_1\subset K_2\subset\dots$ imply that $\bigcap_{k\in\NN}\cG_k=\cG$.

Fix a $k\in\NN$. 
\begin{asparaenum}
\item[\textit{Claim 1:}]
The set $\cG_k$ is open in $\cU$ with respect to the strong operator topology.\\
To this end, let $U\in\cG_k$. Since $K_k$ is compact, there exists an $\eps>0$ such that $d(UK_k,K_k)>\eps$.
Let $\cA := \{ V\in\cU: \sup_{x\in K_k}\norm{(V-U)x}<\eps\}$. Note that $\cA$ is an open neighbourhood of $U$ in the
compact--open topology on $\cU$.
If $V\in \cA$, then
\[
    \norm{Vx-y}\ge \norm{Ux-y}-\norm{(V-U)x} \ge d(UK_k, K_k)-\eps>0
\]
for all $x,y\in K_k$. Therefore $V\in\cG_k$.
So $\cA\subset \cG_k$. 
It is readily verified that the compact--open topology agrees with the strong operator topology on
$\cU$~\cite[Theorem~43.14]{Wil70}. So we have proved the first claim. 

\item[\textit{Claim 2:}]
The set $\cG_k$ is dense in $\cU$ with respect to the uniform operator topology.\\ 
Let $V\in\cU$ and $\eps\in(0,1]$. Set $\delta:=\frac{1}{3k}\sin\eps>0$.
Cover the compact set $VK_k\cup K_k$ by balls $B_\delta(x_1),\dots,B_\delta(x_N)$ with $x_j\in VK_k\cup K_k$ for all
$j\in\{1,\dots,N\}$. Let $p_1,\dots,p_n$ be an orthonormal basis of $L := \linspan\{x_1,\dots,x_N\}$.
Let $q_1,\dots,q_n$ be an orthonormal system in $L^\perp$ and set $M := \linspan\{q_1,\dots,q_n\}$.
Define $W_\eps \in\cU$ such that $W_\eps$ is the 
identity on $(L \oplus M)^\perp$ and $W_\eps$ is the 
rotation of $p_j$ towards $q_j$ by the angle $\eps$ in the 
two-dimensional space $\linspan\{p_j,q_j\}$ for all 
$j\in\{1,\dots,n\}$.
Then
$\norm{W_\eps x_l - x_j}\ge d(W_\eps x_l, L) \ge \norm{x_l}\sin\eps \ge \frac{1}{k}\sin\eps=3\delta$
for all $j,l\in\{1,\dots,N\}$.
Therefore $d(W_\eps VK_k,K_k)\ge\delta$ and $W_\eps V\in\cG_k$.
Moreover,
$\norm{W_\eps V - V} = \norm{W_\eps - I}\le\eps$. We have established the density of $\cG_k$.
\end{asparaenum}

As $\cU$ is a Baire space for the uniform operator topology, by the above and the Baire category theorem 
$\cG = \bigcap_{k\in\NN}\cG_k$ is dense in $\cU$ in the uniform operator topology.
Based on the initial remarks we conclude that $\cG$ is a dense $G_\delta$ set in $\cU$ with respect to the uniform,
strong and weak operator topology.
\end{proof}

Now we give a technical characterisation of when Theorem~\ref{thm:vNeu-op-ran} extends to general Hilbert spaces.
Note that the space occurring in the right hand side of~\eqref{eq:dim-seq-ineq} can be
rewritten using the last identity in Lemma~\ref{lem:or}\,\ref{en:lor-dense}.
\begin{proposition}\label{prop:char}
Let $\cR$ be an operator range in $H$. The following conditions are equivalent:
\begin{romanenum}
\item\label{en:lc-U}
There exists a unitary operator $U$ such that $U\cR\cap\cR=\{0\}$.
\item\label{en:lc-all-rep}
If $\cR$ is represented by $(\cH_n)$, then 
\begin{equation}\label{eq:dim-seq-ineq}
    \dim \bigoplus_{k=1}^n\cH_k \le \dim\Bigl(\cR^\perp\oplus\bigoplus_{k=n+1}^\infty \cH_k\Bigr)
\end{equation}
for all $n\in\NN$.
\item\label{en:lc-one-rep}
There exists a sequence $(\cH_n)$ that represents $\cR$ such that~\eqref{eq:dim-seq-ineq} holds.
\end{romanenum}
\end{proposition}
\begin{proof}\renewcommand{\qedsymbol}{}
`\ref{en:lc-U}$\Rightarrow$\ref{en:lc-all-rep}':
We give a proof by contraposition that uses the same argument as in Example~\ref{ex:ce-vNeu-dom}.
So suppose that~\eqref{eq:dim-seq-ineq} is violated for an $n\in\NN$. 
Replacing $\cH_1$ by $\cH_1\oplus\dots\oplus\cH_n$ and $\cH_k$ by $\{0\}$ for all $k\in\{2,\dots,n\}$,
which does not change the represented operator range,
we may assume without loss of generality that
\begin{equation}\label{eq:dim-simp-ineq}
    \dim \cH_1>\dim\Bigl(\cR^\perp\oplus \bigoplus_{k=2}^\infty \cH_k\Bigr).
\end{equation}
Let $U$ be a unitary operator and set $K := \cR^\perp\oplus \bigoplus_{k=2}^\infty \cH_k$.
Note that $K=\cH_1^\perp$.
By~\eqref{eq:dim-simp-ineq} and Lemma~\ref{lem:cmp-dim} the operator $P_K \restrict{U}{\cH_1}\colon \cH_1\to K$ is not
injective,
where $P_K$ is the orthogonal projection from $H$ onto $K$.
So there exists an $x\in\cH_1\setminus\{0\}$ such that $Ux\in K^\perp=\cH_1$.
In particular, $Ux\in U\cR\cap\cR$ and therefore $U\cR\cap\cR\ne \{0\}$.
Since this holds for every unitary operator $U$, Condition~\ref{en:lc-U} cannot hold.

`\ref{en:lc-all-rep}$\Rightarrow$\ref{en:lc-one-rep}':
Trivial.

`\ref{en:lc-one-rep}$\Rightarrow$\ref{en:lc-U}':
Suppose $\cR$ is represented by the sequence $(\cH_n)$ that satisfies~\eqref{eq:dim-seq-ineq}.
We distinguish two cases.

\textit{Case 1:} Suppose that $\dim H<\aleph_0$.\\
Then there exists an $n_0\in\NN$ such that $\dim\cH_{n}=0$ for all $n>n_0$.
Moreover, $\cR$ is closed and
\[
    \dim\clos{\cR} = \dim\bigoplus_{k=1}^{n_0}\cH_k\le\dim\cR^\perp.
\]
It follows easily (see also the proof of Corollary~\ref{cor:group}) that~\ref{en:lc-U} is satisfied.

\textit{Case 2:}
Suppose that $\dim H\ge\aleph_0$.\\
Before we can proceed with the proof, we need a lemma.
\end{proof}

\begin{lemma}\label{lem:enlarge}
Suppose $\dim H\ge\aleph_0$.
Let $\cR$ be an operator range in $H$.
Suppose there exists a sequence $(\cH_n)$ that represents $\cR$ such that~\eqref{eq:dim-seq-ineq} holds.
Then there exists a dense operator range $\cR'$ that is represented by a sequence $(\cH_n')$
such that $\cR\subset\cR'$ and $\dim\cH_n'=\dim H$ for all $n\in\NN$.
\end{lemma}
\begin{proof}
We first show that there exists a dense operator range $\cR'$ that is represented by $(\cH_n')$ such that
$\cR\subset\cR'$, $\dim\bigoplus_{k=1}^n\cH_k'\le\dim\bigoplus_{k=n+1}^\infty\cH_k'$ and $\dim\cH_n'\ge\aleph_0$ for all
$n\in\NN$.
Since $H$ is infinite-dimensional, it follows from~\eqref{eq:dim-seq-ineq} that
$\dim(\cR^\perp \oplus\bigoplus_{k=n}^\infty\cH_k)\ge\aleph_0$ for all $n\in\NN$.
Hence $\dim\cR^\perp\ge\aleph_0$ or $\{k\in\NN: \cH_k\ne \{0\}\}$ is infinite.
We consider these two cases separately.

\textit{Case 1:} Suppose that $\dim\cR^\perp\ge\aleph_0$.\\
We can decompose $\cR^\perp=\bigoplus_{k=1}^\infty\cK_k$ such that $\dim\cK_n=\dim\cR^\perp$ for all $n\in\NN$.
Then the sequence $(\cH_n\oplus\cK_n)$ of subspaces is orthogonal and represents an operator range $\cR'$.
By Lemma~\ref{lem:or}\,\ref{en:lor-contain} one has $\cR\subset\cR'$.
Moreover, $\cR'$ is dense in $H$ and $\dim(\cH_n\oplus\cK_n)\ge\aleph_0$ for all $n\in\NN$.
Clearly
\[
    \dim\bigoplus_{k=1}^n(\cH_k\oplus\cK_k) \le \dim\bigoplus_{k=n+1}^\infty (\cH_k\oplus\cK_k)
\]
for all $n\in\NN$.

\textit{Case 2:} Suppose that $M := \{k\in\NN: \cH_k\ne \{0\}\}$ is infinite and $\dim\cR^\perp<\aleph_0$.\\
By replacing $\cH_1$ with $\cH_1\oplus\cR^\perp$ we may assume that
 $\cR$ is dense in $H$ and
\begin{equation}\label{eq:dim-ineq2}
    \dim \bigoplus_{k=1}^n\cH_k \le \dim\bigoplus_{k=n+1}^\infty \cH_k
\end{equation}
for all $n\in\NN$.
Write $M$ as the countable disjoint union of sets $(M_n)_{n\in\NN}$ such
that $\card M_n =\card M=\aleph_0$ for all $n\in\NN$.
Set $M_n' := M_n\setminus\{1,\dots,n-1\}$ for all $n\in\NN$.
Then $\card M_n'=\aleph_0$.
Set $M':=\bigcup_{n\in\NN}M_n'$.
For all $n\in M'$ let $\cK_n$ be a fixed one-dimensional subspace of $\cH_n$.
Set $\cK_n' := \bigoplus_{k\in M_n'} \cK_k$ for all $n\in\NN$. Then $\dim\cK_n'=\aleph_0$.
For all $n\in\NN$ define
\[
    \cH_n' := \begin{cases}
        (\cH_n\ominus\cK_n)\oplus\cK_n' & \text{if $n\in M'$,} \\
        \cH_n\oplus\cK_n' & \text{if $n\in\NN\setminus M'$.}
    \end{cases} 
\]
Then $\dim\cH_n'=\dim\cH_n + \aleph_0$ for all $n\in\NN$.
Clearly $\cH_n\subset\cH_n'\subset\bigoplus_{k=1}^n\cH_k'$ if $n\in\NN\setminus M'$.
On the other hand, if $n\in M'$, then there exists a (unique) $m\in\NN$ such that $n\in M_m'$.
Then $n\notin\{1,\dots,m-1\}$, and hence $n\ge m$. Therefore
$\cK_n\subset\cK_m'\subset\cH_m'\subset\bigoplus_{k=1}^n\cH_k'$. 
Thus also in this case $\cH_n\subset\bigoplus_{k=1}^n\cH_k'$.
The sequence of subspaces $(\cH_n')$ is orthogonal since the subspaces
\begin{alignat*}{2}
    \cH_n\ominus\cK_n &\quad\text{with $n\in M'$,}\\
    \cH_n &\quad\text{with $n\in\NN\setminus M'$ and}\\
    \cK_n &\quad\text{with $n\in M'$}
\end{alignat*}
are orthogonal, together with the disjointness of the $M_n'$.
It follows that the operator range $\cR'$ represented by $(\cH_n')$ contains $\cR$ by Lemma~\ref{lem:or}\,\ref{en:lor-contain}. Moreover, 
\[
    \dim\bigoplus_{k=1}^n\cH_k' = \aleph_0+\dim\bigoplus_{k=1}^n\cH_k
    \le \aleph_0 + \dim\bigoplus_{k=n+1}^\infty\cH_k = \dim\bigoplus_{k=n+1}^\infty\cH_k'
\]
for all $n\in\NN$ by~\eqref{eq:dim-ineq2}.

Thus we may assume that $\cR$ is a dense operator range with $\dim\cH_n\ge\aleph_0$ for all $n\in\NN$ after enlarging $\cR$ appropriately. We now continue the argument under this assumption.
Decompose $\cH_n=\bigoplus_{k=1}^n \cK^{(n)}_k$
such that $\dim\cK^{(n)}_k=\dim\cH_n$ for all $n\in\NN$ and $k\in\{1,\dots,n\}$.
Set $\cH'_k \coloneqq \bigoplus_{m=k}^\infty \cK^{(m)}_{k}$
for all $k\in\NN$.
Then 
\[
    \cH_n = \bigoplus_{k=1}^n\cK^{(n)}_k \subset\bigoplus_{m=1}^\infty\bigoplus_{k=1}^{\min\{n,m\}} \cK^{(m)}_k =
\bigoplus_{k=1}^n\bigoplus_{m=k}^\infty \cK^{(m)}_k = \bigoplus_{k=1}^n\cH_k'.
\]
So $(\cH'_n)$ represents an operator range $\cR'$ that contains $\cR$ by Lemma~\ref{lem:or}\,\ref{en:lor-contain}.
Moreover, 
\begin{align*}
    \dim H=\dim\bigoplus_{n=1}^\infty\cH_n=\dim\bigoplus_{n=k}^\infty\cH_n
    = \dim\bigoplus_{n=k}^\infty\cK^{(n)}_k = \dim\cH_k'
\end{align*}
for all $k\in\NN$, where we used~\eqref{eq:dim-seq-ineq} in the second step.
\end{proof}

We complete the proof of Proposition~\ref{prop:char}.
\begin{proof}[End of the proof of Proposition~\ref{prop:char}]
By Lemma~\ref{lem:enlarge} it suffices to find a unitary operator $U$ for an operator range $\cR$ that is represented by $(\cH_n)$ such
that $\kappa \coloneqq\dim H = \dim\cH_n$ for all $n\in\NN$.
Let $S$ be the Hilbert space direct sum of $\kappa$ disjoint copies of the operator $T$ from
Example~\ref{ex:T-nc-range}.
By Lemma~\ref{lem:or}\,\ref{en:lor-unequi} there exists a unitary operator $W$ such that $W\rg S = \cR$. 
Moreover, since Proposition~\ref{prop:israel} applies to $T$, there exists
a unitary operator $V$ such that $V\rg S\cap \rg S=\{0\}$.
Then $U=WVW^*$ satisfies
\[
    U\cR\cap \cR = WVW^*W\rg S \cap W\rg S = W(V\rg S\cap\rg S) = \{0\}.
\] 
The proof of Proposition~\ref{prop:char} is complete.
\end{proof}

The following corollary gives a sufficient condition for when two operator ranges $\cR$ and $\cS$ admit a unitary
operator $W$
such that $W\cR\cap\cS=\{0\}$. 
\begin{corollary}
Let $H$ be an infinite-dimensional Hilbert space.
Suppose that $\cR$ and $\cS$ are operator ranges such that there exist unitary operators $U$ and $V$ with
$U\cR\cap\cR=\{0\}$ and $V\cS\cap\cS=\{0\}$.
Then there exists a unitary operator $W$ such that $W\cR\cap\cS=\{0\}$.
\end{corollary}
\begin{proof}
By Proposition~\ref{prop:char} and Lemma~\ref{lem:enlarge} we may suppose without loss of generality,
possibly enlarging $\cR$ and $\cS$, that
$\cR$ is represented by $(\cH_n)$, the operator range $\cS$ is represented by $(\cK_n)$ and $\dim\cH_n=\dim\cK_n=\dim H$ for all
$n\in\NN$. Then by Lemma~\ref{lem:or}\,\ref{en:lor-unequi} there exists a unitary operator $Z$ such that $Z\cR =\cS$.
So the unitary operator $W := VZ$ has the desired property.
\end{proof}

The following is inspired by the proof of von Neumann's theorem as presented in~\cite[Theorem~3.6]{FW71:op-rg}.
Note that the next proposition is also applicable to the counterexample in Example~\ref{ex:ce-vNeu-dom}.
\begin{proposition}
Let $H$ be an infinite-dimensional Hilbert space.
Suppose that $\cR$ is a dense operator range in $H$.
Then there exists an operator range $\cS\subset\cR$ and a unitary operator $U$
such that $\cS$ is dense in $H$ and $U\cS\cap\cS=\{0\}$.
\end{proposition}
\begin{proof}
\textit{Case 1:} Suppose that $\cR$ is represented by $(\cH_n)$ satisfying $\dim\cH_n<\aleph_0$ for all $n\in\NN$.\\
Then Proposition~\ref{prop:char} directly applies to $\cR$.

\textit{Case 2:} Suppose that $\cR$ is represented by $(\cH_n)$ and that there exists an $n_0\in\NN$ such that
$\dim\cH_{n_0}\ge\aleph_0$.\\
Clearly $\cR$ is equal to the operator range represented by the orthogonal decomposition
$\bigl(\bigoplus_{k=1}^{n_0}\cH_k, \cH_{n_0+1}, \cH_{n_0+2},\dots\bigr)$.
Hence we may assume that $\dim\cH_1\ge\aleph_0$.
Let $(\cK_n)$ be an orthogonal decomposition of $H_1$ such that $\dim\cK_n=\dim\cH_1$ for all $n\in\NN$.
Then the operator range $\cR'$ represented by the orthogonal decomposition $(\cK_1,\cH_2\oplus \cK_2,\cH_3\oplus\cK_3,\ldots)$
is dense in $H$ and satisfies $\cR'\subset\cR$ by Lemma~\ref{lem:or}\,\ref{en:lor-dense} and~\ref{en:lor-contain}.
Therefore we may assume that the orthogonal sequence $(\cH_n)$ representing $\cR$ satisfies $\dim\cH_n\ge\aleph_0$ for all $n\in\NN$.

For all $n\in\NN$ we can decompose $\cH_n$ into a countable orthogonal direct sum of $(\cH^{(n)}_k)_{k\in\NN}$
such that $\dim\cH^{(n)}_k=\dim\cH_n$ for all $k\in\NN$.
Define $\cK_n \coloneqq \bigoplus_{k=1}^n\cH^{(k)}_{n-k+1}$ for all $n\in\NN$.
Then the operator range $\cS$ represented by $(\cK_n)$ is dense by
Lemma~\ref{lem:or}\,\ref{en:lor-dense} and satisfies
$\cS\subset\cR$ by Lemma~\ref{lem:or}\,\ref{en:lor-contain}. Moreover, $\dim\cK_n\le\dim\cK_{n+1}$ for all $n\in\NN$.
So there exists a unitary operator $U$ such that $U\cS\cap\cS=\{0\}$ by
Proposition~\ref{prop:char}\,\ref{en:lc-one-rep}$\Rightarrow$\ref{en:lc-U}.
\end{proof}

While the more technical Proposition~\ref{prop:char} is already useful by itself,
we use it to provide the following reformulation of von Neumann's theorem that holds for general Hilbert spaces.
\begin{theorem}\label{thm:vNeu-nonsep}
Let $\cR$ be an operator range in $H$. 
The following are equivalent:
\begin{romanenum}
\item 
There exists a unitary operator $U$ such that $U\cR\cap \cR=\{0\}$.
\item\label{en:cond-K}
For every closed subspace $K\subset \cR$ one has
\[
    \dim K \le \dim K^\perp.
\]
\end{romanenum}
\end{theorem}
\begin{proof}
Suppose that there exists a unitary operator $U$ such that $U\cR\cap\cR=\{0\}$.
Let $K\subset\cR$ be a closed subspace. Note that $\cR = K\oplus (\cR\cap K^\perp)$
and that $\cR\cap K^\perp$ is an operator range by Proposition~\ref{prop:opran}\,\ref{en:or-is}.
Suppose that $(\cK_n)$ represents $\cR\cap K^\perp$. 
Define $\cH_1:=K$, $\cH_2:=\cK_1\oplus\cK_2$ and $\cH_n := \cK_n$ for all $n\ge 3$.
It is readily verified that $(\cH_n)$ represents $\cR$.
Then by Proposition~\ref{prop:char}\,\ref{en:lc-U}$\Rightarrow$\ref{en:lc-all-rep} with $n=1$ and Lemma~\ref{lem:or}\,\ref{en:lor-dense} one obtains $\dim K\le\dim K^\perp$.

For the converse direction, we give a proof by contraposition.
Let $(\cH_n)$ represent $\cR$
and suppose that there does not exist a unitary operator $U$ such that $U\cR\cap\cR=\{0\}$.
By Proposition~\ref{prop:char} there exists an $n_0\in\NN$ such that
\[
    \dim\bigoplus_{k=1}^{n_0}\cH_k > \dim\Bigl(\cR^\perp\oplus\bigoplus_{k=n_0+1}^\infty\cH_k\Bigr).
\]
Set $K := \bigoplus_{k=1}^{n_0}\cH_k$. Then $K$ is closed, $K\subset\cR$ and $\dim K > \dim K^\perp$.
\end{proof}

\begin{remark}
If $H$ is infinite-dimensional, then Condition~\ref{en:cond-K} in Theorem~\ref{thm:vNeu-nonsep} is equivalent
to requiring $\dim K^\perp=\dim H$ for every closed subspace $K\subset\cR$.

Theorem~\ref{thm:vNeu-nonsep} extends the separable case covered in Theorem~\ref{thm:vNeu-op-ran}.
In fact, suppose that $\dim H=\aleph_0$ and $K\subset\cR$ is closed with $\dim K^\perp<\aleph_0$.
Then $\cR\cap K^\perp$ is closed and therefore $\cR=K\oplus (\cR\cap K^\perp)$ is closed.
Consequently Condition~\ref{en:cond-K} is clearly satisfied 
for a nonclosed operator range in the separable Hilbert space $H$.

We point out that it is allowed in Theorem~\ref{thm:vNeu-nonsep} that $H$ is finite-dimensional, that $\cR$ is closed
and that $\cR$ is not dense.
\end{remark}

The following is a straightforward reformulation for domains of closed operators.
\begin{corollary}
Let $H$ be a Hilbert space and $A$ be a densely defined closed operator in $H$.
Suppose that $\dim K\le\dim K^\perp$ for every closed subspace $K\subset D(A)$.
Then there exists a unitary operator $U$ such that $D(U^*AU)\cap D(A)=\{0\}$.
\end{corollary}

Analogously to~\cite[Corollary~1 of Theorem~3.6]{FW71:op-rg},
we are able to obtain the following corollary by inspection of the proof of Proposition~\ref{prop:char}.
\begin{corollary}\label{cor:group}
Let $\cR$ be an operator range in $H$.
Suppose that there exists a unitary operator $U$ such that $U\cR\cap\cR=\{0\}$.
Then there exists a uniformly continuous unitary group $(U_t)_{t\in\RR}$ and an uncountable interval $I\subset\RR$ such that
$U_t\cR\cap U_s\cR=\{0\}$ for all $t,s\in I$ with $t\ne s$.
\end{corollary}
\begin{proof}
As in the proof of Proposition~\ref{prop:char}\,\ref{en:lc-one-rep}$\Rightarrow$\ref{en:lc-U} we distinguish two cases.

\textit{Case 1:} Suppose that $H$ is infinite-dimensional.\\
If $T$ is the operator as in Example~\ref{ex:T-nc-range}, then one can argue as in~\cite[Corollary~1 of
Theorem~3.6]{FW71:op-rg}.
So there exists a uniformly continuous group $(V_t)_{t\in\RR}$ such that $V_t\rg T \cap V_s\rg T=\{0\}$
for all $t,s\in\RR$ with $t\ne s$.
Then the claim for the first case follows by taking a direct sum of the appropriate cardinality as in the last part of the proof of
Proposition~\ref{prop:char}.

\textit{Case 2:} Suppose that $H$ is finite-dimensional.\\
Then $\cR$ is closed and $\dim\cR\le\dim\cR^\perp$.
Let $p_1,\dots,p_n$ be an orthonormal basis of $\cR$ and
$q_1,\dots,q_n$ an orthonormal system in $\cR^\perp$.
Define $U_t\in\cU$ such that $U_t$ is the 
identity on $\linspan\{p_1,\dots,p_n,q_1,\dots,q_n\}^\perp$ and $U_t$ is the 
rotation of $p_j$ towards $q_j$ by the angle $t$ in the 
two-dimensional space $\linspan\{p_j,q_j\}$ for all 
$j\in\{1,\dots,n\}$.
Then $(U_t)_{t\in\RR}$ defines a uniformly continuous group of unitary operators on $H$.
Moreover, the claim is satisfied with $I=[0,\pi)$.
\end{proof}
\begin{remark}\label{rem:fdim-fam}
It follows from the proof of Corollary~\ref{cor:group} that one can actually choose $I=\RR$ if $H$ is infinite-dimensional.
In general this is not possible in the finite-dimensional case. In fact, suppose $H=\CC^2$ and let $(U_t)_{t\in\RR}$ be
a uniformly continuous unitary group. By Stone's theorem and after a unitary transformation we may assume that
there exist $\lambda_1,\lambda_2\in\RR$ such that
\[
    U_t=\begin{pmatrix} e^{it\lambda_1} & 0 \\ 0 & e^{it\lambda_2} \end{pmatrix}
\]
for all $t\in\RR$.
Let $t\in\RR\setminus\{0\}$ be such that $t(\lambda_1-\lambda_2)\in 2\pi\ZZ$.
Then $U_t = e^{it\lambda_2} U_0$.

We point out, however, that in the finite-dimensional case we can find a continuous family of \emph{unitary and
self-adjoint} operators
$(U_t)_{t\in I}$ such that $U_t\cR\cap U_s\cR=\{0\}$ for all $t,s\in I$ with $t\ne s$. 
In fact, it suffices to consider the case $H=\CC^2$ and $\cR=\linspan\{e_1\}$.
Then the unitary operators
\[
    U_t := \begin{pmatrix} \cos t & -\sin t\\ \sin t & \cos t\end{pmatrix}
    \begin{pmatrix} 1 & 0 \\ 0 & -1 \end{pmatrix}
    \begin{pmatrix} \cos t & \sin t \\ -\sin t & \cos t\end{pmatrix}
\]
for all $t\in I :=[0,\frac{\pi}{2})$ have the desired properties.
\end{remark}

\section{Stability and density}\label{sec:stab}

We use Theorem~\ref{thm:vNeu-nonsep} to prove the following stability result.
\begin{theorem}\label{thm:clos}
Let $H$ be a Hilbert space. Then the set of operators $T\in\Linop(H)$ that admit a unitary operator $U$ 
such that $U\rg T\cap\rg T=\{0\}$ is closed with respect to the uniform operator norm.
\end{theorem}
\begin{proof}
We prove that the complement of the above set of operators is open with respect to the uniform operator norm.
To this end, we argue along the lines of the proof for the openness of the set of Fredholm operators in
\cite[Theorem~XVII.2.3]{Serge93}.

Suppose that $T\in\Linop(H)$ does not admit a unitary operator $U$ such that $U\rg T\cap\rg T=\{0\}$.
By Theorem~\ref{thm:vNeu-nonsep} there exists a closed subspace $K$ of $\rg T$ such that $\dim K > \dim K^\perp$.
Define $W := (\restrict{T}{(\ker T)^\perp})^{-1}(K)$. Then $W$ is a closed subspace of $(\ker T)^\perp$ and $K=TW$.

For all $S\in\Linop(H)$ define $\widehat{S}\colon W\times K^\perp \to H$ by $\widehat{S}(x,h)=Sx+h$.
Obviously $\widehat{S}$ is a bounded operator. Moreover, note that
\begin{equation}\label{eq:Shat-est}
    \norm{\widehat{S}_1-\widehat{S}_2}_{\Linop(W\times K^\perp;H)}=\sup_{\norm{x}^2+\norm{h}^2\le 1}\norm{S_1x+h - S_2x - h}
       \le\norm{S_1-S_2}_{\Linop(H)}
\end{equation}
for all $S_1,S_2\in\Linop(H)$.
We claim that $\widehat{T}$ is an isomorphism. Obviously $\rg\widehat{T}=H$, so $\widehat{T}$ is surjective.
Suppose that $\widehat{T}(x,h)=0$ for an $x\in W$ and an $h\in K^\perp$. Then $Tx=-h$ with $Tx\in K$ and $-h\in
K^\perp$. So $h=Tx=0$. As $x\in W\subset(\ker T)^\perp$, one obtains $(x,h)=0$. Hence $\widehat{T}$ is injective.

As the set of isomorphisms between $W\times K^\perp$ and $H$ is open with respect to the uniform operator norm,
by~\eqref{eq:Shat-est} there exists an $\eps>0$ such that for all $S\in\Linop(H)$ with $\norm{T-S}_{\Linop(H)}<\eps$ 
the operator $\widehat{S}$ is an isomorphism.
Let $S\in\Linop(H)$ with $\norm{T-S}_{\Linop(H)}<\eps$.
It remains to show that $S$ does not admit a unitary operator $U$ such that $U\rg S\cap\rg S=\{0\}$.
By Theorem~\ref{thm:vNeu-nonsep} it suffices to show that $SW$ is a closed subspace of $\rg S$ with $\dim
SW>\dim(SW)^\perp$.
First observe that $SW = \widehat{S}(W\times\{0\})$ is closed in $H$.
Moreover, as $SW + K^\perp = H$, it follows that $(SW)^\perp = P(K^\perp)$, where $P$ is the orthogonal projection from $H$ onto $(SW)^\perp$.
Hence $\dim(SW)^\perp\le\dim K^\perp$ by Lemma~\ref{lem:cmp-dim}. Using Lemma~\ref{lem:cmp-dim} again one deduces that
\[
    \dim(SW)^\perp\le\dim K^\perp<\dim K=\dim\widehat{T}(W\times\{0\}) = \dim W = \dim\widehat{S}(W\times\{0\})=\dim SW.
\]
The proof is complete.
\end{proof}

By Proposition~\ref{prop:israel} every compact operator $T$ on an infinite-dimensional
separable Hilbert space admits a unitary operator $U$ such that $U\rg T\cap\rg T=\{0\}$.
Using the previous theorems and the results in~\cite{Luft68}, we shall prove that, in the setting of an arbitrary
infinite-dimensional Hilbert space, the same holds for 
every operator in any closed two-sided ideal of $\Linop(H)$ that is different from
$\Linop(H)$. 

To this end, for every cardinal $\kappa$ one defines
\[
    \cF_\kappa(H) := \{ T\in\Linop(H) : \dim\clos{\rg T}<\kappa \}
\]
and $\cC_\kappa(H) := \clos{\cF_\kappa(H)}$, where the closure is taken in $\Linop(H)$ with respect to the uniform operator
norm. Then $\cC_\kappa(H)$ is a closed two-sided $*$-ideal of $\Linop(H)$ by~\cite[Corollary~5.2]{Luft68}.
\begin{corollary}
Let $H$ be an infinite-dimensional Hilbert space.
Then for all $T\in\cC_{\dim H}(H)$ there exists a unitary operator $U$ such that $U\rg T\cap\rg T =\{0\}$.
\end{corollary}
\begin{proof}
By Theorem~\ref{thm:vNeu-nonsep} for all $T\in\cF_{\dim H}(H)$ there exists a unitary operator $U$
such that $U\rg T\cap\rg T =\{0\}$. Now the claim follows from Theorem~\ref{thm:clos}.
\end{proof}

We show that if $\cS$ is an operator range in an infinite-dimensional Hilbert space that
admits a unitary operator $U$ with $U\cS\cap\cS=\{0\}$, then any `compact perturbation' of $\cS$ has the same property.
\begin{theorem}\label{thm:sum-pert}
Let $H$ be an infinite-dimensional Hilbert space, $\cR$ the range of an operator in $\cC_{\dim H}(H)$ and
$\cS$ an operator range in $H$ such that there exists a unitary operator $U$ with $U\cS\cap\cS=\{0\}$.
Then there exists a unitary operator $W$ such that $W(\cR+\cS)\cap(\cR+\cS)=\{0\}$.
\end{theorem}
\begin{proof}
We give a proof by contraposition. So suppose that $K\subset\cR+\cS$ is a closed subspace with $K^\perp<\dim H$.
By Theorem~\ref{thm:vNeu-nonsep} it suffices to prove that there exists a closed subspace $M_2\subset\cS$ such that
$\dim M_2^\perp<\dim H$.

Clearly $(\cR+K^\perp) + \cS=H$, where also $\cR+K^\perp$ is an operator range by
Proposition~\ref{prop:opran}\,\ref{en:or-sum}. So it follows from Proposition~\ref{prop:opran}\,\ref{en:or-sum-all} that there exist
closed subspaces $M_1,M_2$ of $H$ such that $M_1\subset\cR+K^\perp$, $M_2\subset\cS$, $M_1\cap M_2=\{0\}$ and
$M_1+M_2=H$. 
Let $T\in\cC_{\dim H}(H)$ be such that $\cR=\rg T$. Then $\cR = \rg(TT^*)^{1/2}$ by Douglas' lemma.
It follows from the ideal property of $\cC_{\dim{H}}(H)$ and uniform approximation of the square root
that also $(TT^*)^{1/2}\in\cC_{\dim H}(H)$. So we may assume that $\cR=\rg T$ with $T\in\cC_{\dim H}(H)$ positive.
By Proposition~\ref{prop:opran}\,\ref{en:or-sum} the operator $A := (T^2 + P_{K^\perp})^{1/2}$ has range $\cR+K^\perp$,
where $P_{K^\perp}$ is the orthogonal projection onto $K^\perp$.
As before we obtain $A\in\cC_{\dim H}(H)$.
It follows from~\cite[Theorem~5.1]{Luft68} that $\dim M_1 <\dim H$ as $M_1$ is a closed subset of $\rg A=(\cR+K^\perp)$.
Moreover,
\[
    M_2^\perp  = P_{M_2^\perp}(M_1+M_2) = P_{M_2^\perp} M_1,
\]
where $P_{M_2^\perp}$ is the orthogonal projection onto $M_2^\perp$.
Therefore $\dim M_2^\perp=\dim(P_{M_2^\perp}M_1)\le\dim M_1<\dim H$, where we used Lemma~\ref{lem:cmp-dim} for the first
inequality.
This concludes the proof.
\end{proof}

The following is an immediate consequence of Theorem~\ref{thm:sum-pert} and strengthens the density result in
Proposition~\ref{prop:israel}.
\begin{corollary}\label{cor:dense-cpt-per}
Let $H$ be an infinite-dimensional Hilbert space, $\cR$ the range of an operator in $\cC_{\dim H}(H)$ and
$\cS$ an operator range in $H$ such that there exists a unitary operator $V$ with $V\cS\cap\cS=\{0\}$.
Then the set
\[
    \cG := \{U\in\cU : U\cR\cap\cS =\{0\}\}
\]
is dense in $\cU$ with respect to the uniform operator norm.
\end{corollary}
\begin{proof}
Let $\eps>0$ and $V\in\cU$.
By Theorem~\ref{thm:sum-pert} there exists a unitary operator $W$ such that $W(\cR +V^*\cS)\cap(\cR+V^*\cS)=\{0\}$. 
It follows from Corollary~\ref{cor:group} that we can ensure $\norm{I-W}<\eps$.
Alternatively, the latter also follows in a more self-contained way from Proposition~\ref{prop:israel} 
and the construction at the end of the proof of Proposition~\ref{prop:char}.
Hence $W\cR \cap V^*\cS=\{0\}$, or equivalently with $U := VW$, one has $U\cR\cap\cS=\{0\}$.
As $\norm{V-U} = \norm{I-W}<\eps$, the claim follows.
\end{proof}

In~\cite[bottom of p.\,229]{vNeu29} von Neumann gives the following interpretation of his findings related to
Theorem~\ref{thm:vNeu-dom}:
\begin{quote}
Noch eher sind diejenigen unit\"aren Matrizen als pathologisch zu bezeichnen die unsere paradoxen [\dots] 
\"Aquivalenzen vermitteln, trotzdem gerade diese beschr\"ankt sind!
\end{quote}
So he attributes the `pathological' phenomenon in Theorem~\ref{thm:vNeu-dom} to the richness of the unitary operators.
Looking at Corollary~\ref{cor:dense-cpt-per} or Israel's result in Proposition~\ref{prop:israel}, one might -- in this
line of thought -- expect that
if $\cR$ is an operator range that admits a unitary operator $U$ such that $U\cR\cap\cR=\{0\}$ one automatically has
that the set
\begin{equation}\label{eq:set-G}
    \cG = \{U\in\cU : U\cR\cap\cR=\{0\}\}
\end{equation}
is dense in the unitary operators with respect to the uniform operator norm.
Somewhat surprisingly, we shall prove that this expectation is unfounded.

We need the following lemma, which is inspired by~\cite[Theorem~2.4]{FW71:op-rg}.
\begin{lemma}\label{lem:dense-impos}
Let $H$ be an infinite-dimensional Hilbert space, $\cR$ a dense operator range and $V$ a unitary operator such that $\cR+V\cR=H$.
Then there exists an $\eps>0$ such that for all unitary operators $W$ with $\norm{I-W}<\eps$ one has
$\cR+WV\cR=H$.
\end{lemma}
\begin{proof}
Let $T\in\Linop(H)$ be a positive operator such that $\rg T=\cR$.
By the assumption and Proposition~\ref{prop:opran}\,\ref{en:or-sum} the positive operator $(T^2 + VT^2V^*)^{1/2}$ has
range $H$ and therefore is invertible.
So there exists a $\delta>0$ such that
\begin{equation}\label{eq:sum-H-lb}
    \delta^2\norm{x}^2 \le \scalar{(T^2 + VT^2V^*)x}{x} = \norm{Tx}^2 + \norm{TV^*x}^2
\end{equation}
for all $x\in H$.
Set $\eps := \frac{\delta}{2\norm{TV^*}}>0$. Let $W$ be a unitary operator such that $\norm{I-W}<\eps$.
Then
\begin{align*}
    \norm{TV^*x} &\le \norm{TV^*W^*x} + \norm{TV^*(I-W^*)x} \\
        &\le \norm{TV^*W^*x} +\norm{TV^*}\norm{I-W}\norm{x} \\
        &\le\norm{TV^*W^*x} + \frac{\delta}{2}\norm{x}
\end{align*}
By plugging the above into~\eqref{eq:sum-H-lb} we obtain
\[
    \delta^2\norm{x}^2\le\norm{Tx}^2 + 2\norm{TV^*W^*x}^2 + 2\frac{\delta^2}{4}\norm{x}^2
\]
and therefore
\begin{equation}\label{eq:low-clos}
    \frac{\delta^2}{4}\norm{x}^2\le\norm{Tx}^2 + \norm{WVTV^*W^*x}^2 = \scalar{(T^2 + WVT^2V^*W^*)x}{x}
\end{equation}
for all $x\in H$.
As $\rg(T^2+WVT^2V^*W^*)^{1/2}=\cR+WV\cR$ by Proposition~\ref{prop:opran}\,\ref{en:or-sum},
it follows from~\eqref{eq:low-clos} that $\cR+WV\cR$ is closed.
Since $\cR$ is dense, one has $\cR+WV\cR=H$.
\end{proof}

The following example shows that in general density of $\cG$ in~\eqref{eq:set-G} cannot even be expected in the separable case.
\begin{example}\label{ex:ce-dens}
Suppose that $H$ is an infinite-dimensional Hilbert space such that
\[
    H = \bigoplus_{k=2}^\infty \cK_k\oplus\bigoplus_{k=2}^\infty\cH_k
\]
with $\dim\cK_k=\dim\cH_k=\dim H$ for all $k\in\NN\setminus\{1\}$. Set $\cH_1 := \bigoplus_{k=2}^\infty\cK_k$
and $\cK_1 := \bigoplus_{k=2}^\infty\cH_k$. 
Clearly one can choose $H$ to be separable.

Let $\cR$ be the operator range represented by $(\cH_n)$ and $\cS$ be the operator range represented by $(\cK_n)$.
By Lemma~\ref{lem:or}\,\ref{en:lor-unequi} there exists a unitary operator $V$ such that $V\cS=\cR$.
Moreover, by Proposition~\ref{prop:char} there exists a unitary operator $U$ such that $U\cR\cap\cR=\{0\}$.

Observe that $\cR+V\cR = \cR + \cS = H$. It follows from Lemma~\ref{lem:dense-impos} that there exists an $\eps>0$ such that
$\cR + WV\cR = H$ for all unitary operators $W$ with $\norm{I-W}<\eps$.
But if $W$ is a unitary operator such that $\cR+WV\cR=H$, then it not possible that $\cR\cap WV\cR=\{0\}$ by
Lemma~\ref{prop:opran}\,\ref{en:or-sum-clos} and because $\cR$ is not closed.
We have proved that $V$ is not in the closure of the set $\cG$ in~\eqref{eq:set-G} with respect to the uniform operator norm.
\end{example}

\begin{remark}
Several recent results on operator ranges in~\cite{AZ2015} can immediately be extended to the nonseparable case
provided the corresponding operator range satisfies Condition~\ref{en:cond-K} in Theorem~\ref{thm:vNeu-nonsep}.
In particular, this applies for example to~\cite[Theorems~3.7, 3.12 and 3.19]{AZ2015}.
As a consequence of these results, if $\cR$ is an operator range in $H$ such that there exists a unitary operator
$U$ with $U\cR\cap\cR=\{0\}$, then there exist uncountably many such operators $U$ that are both \emph{unitary and
self-adjoint}.
For the finite-dimensional case, we pointed this out in the last part of Remark~\ref{rem:fdim-fam}.
\end{remark}

\subsection*{Acknowledgements} 

The second named author is
most grateful for the hospitality extended to him during a research visit at the University of Auckland.
Part of this work was supported by the Marsden Fund Council from Government
funding, administered by the Royal Society of New Zealand.

\microtypesetup{disable}
\providecommand{\bysame}{\leavevmode\hbox to3em{\hrulefill}\thinspace}

\end{document}